\documentclass[aap,preprint]{imsart}
\usepackage{amsthm,amsmath}
\RequirePackage[colorlinks,citecolor=blue,urlcolor=blue]{hyperref}

\startlocaldefs
\usepackage[T1]{fontenc}
\usepackage[margin=3cm,footskip=1cm]{geometry}
\usepackage{amsmath,amssymb,amsthm,bm,mathtools,xspace,booktabs,url}
\usepackage{graphicx,etoolbox}
\usepackage{ dsfont }
\usepackage{hyperref}
\usepackage{bbold}
\usepackage{tikz}
\usepackage{ mathrsfs }
\usepackage[shortlabels]{enumitem}
\usepackage{xcolor}
\usepackage{color}

\newcommand{\dd}{\mathrm{d}}
\makeatletter
\global\let\tikz@ensure@dollar@catcode=\relax
\makeatother
\setlist{
  listparindent=\parindent,
  parsep=0pt,
}
\usepackage{amssymb}

\AtBeginEnvironment{thebibliography}{
  \small
  \setlength\itemsep{0.75em plus 0.5em minus 0.75em}%
}

\usepackage{caption}
\usepackage[raggedright]{titlesec}
\numberwithin{equation}{section}
\usepackage{amsfonts}
\theoremstyle{plain} 
\newtheorem{theorem}{Theorem}[section]
\newtheorem{Lemma}[theorem]{Lemma}
\newtheorem{Proposition}[theorem]{Proposition}

\theoremstyle{definition} 

\newtheorem{Remark}[theorem]{Remark}

\usepackage[above,below,section]{placeins}
\usepackage[all]{onlyamsmath}
\usepackage{multirow}

\usepackage{xcolor}
\usepackage{color}

\usepackage{mathrsfs}
\usepackage{pdfpages}
\definecolor{darkmagenta}{rgb}{0.5,0,0.5}
\definecolor{darkgreen}{rgb}{0,0.6,0}
\definecolor{darkblue}{rgb}{0,0,0.6}
\definecolor{darkred}{rgb}{0.8,0,0}
\definecolor{mellow}{rgb}{.847, 0.72, 0.525}

\endlocaldefs

\begin{document}

\begin{frontmatter}

\title{Rate of Strong Convergence to Markov-modulated Brownian motion}
\runtitle{Rate of strong convergence to MMBM}

\begin{aug}
  \author{\fnms{Giang T.}  \snm{Nguyen}\thanksref{t2}\ead[label=e1]{giang.nguyen@adelaide.edu.au}}
  \and
  \author{\fnms{Oscar} \snm{Peralta}\thanksref{t1}\corref{}\ead[label=e2]{oscar.peraltagutierrez@adelaide.edu.au}}
  \thankstext{t2}{Supported by ARC Grant DP180103106. }
  \thankstext{t1}{Corresponding author. Supported by ARC Grant DP180103106. }
  \runauthor{Nguyen et al.}

  \affiliation{The University of Adelaide}

  \address{School of Mathematical Sciences\\ The University of Adelaide\\ SA 5000, Australia\\
\printead{e1}\\
\phantom{E-mail:\ }\printead*{e2}}
\end{aug}

\begin{abstract}
In \cite{latouche2015morphing}, the authors constructed a sequence of stochastic fluid processes and showed that it converges weakly to a Markov-modulated Brownian motion (MMBM). Here, we construct a different sequence of stochastic fluid processes and show that it converges strongly to an MMBM. To the best of our knowledge, this is the first result on strong convergence to a Markov-modulated Brownian motion. 

We also prove that the rate of this almost sure convergence is $o(n^{-1/2} \log n)$. When reduced to the special case of standard Brownian motion, our convergence rate is an improvement over that obtained by a different approximation in \cite{gorostiza1980rate}, which is $o(n^{-1/2}(\log n)^{5/2})$.
\end{abstract}



\end{frontmatter}
\section{Introduction.}
\label{sect:intro} 

The family of \emph{flip-flop} processes
corresponds to a class of piecewise-linear Markov processes that converges, in some sense, to a standard Brownian motion. 
%
%
Specifically, for $\lambda>0$, let $ =\{\varphi^\lambda(t) \}_{t\ge 0}$
%
be a Markov jump process with state space $\{+,-\}$, initial distribution $(1/2, 1/2)$
%
and intensity matrix
	\begin{align*} 
		\left[\begin{array}{rr} 
		-\lambda & \lambda\\
		\lambda & -\lambda
		\end{array} \right]. 
	\end{align*}
Let $r(+) = \sqrt{\lambda}$, $r(-)=-\sqrt{\lambda}$ and define
%
%
\begin{align}
	\label{eq:flipflop1}
		F^\lambda(t)  = \int_0^t r(\varphi^\lambda(s))\dd s,\quad t\ge 0.
\end{align}
We call $\{(F^\lambda(t), \varphi^\lambda(t))\}_{t\ge 0}$ a flip-flop process. It can be shown (see, e.g.,  \cite{ramaswami2013fluid}) that  $\mathcal{F}^\lambda=\{F^\lambda(t)\}_{t\ge 0}$ converges \emph{weakly} to a standard Brownian motion $\mathcal{B} =\{B(t)\}_{t\ge 0}$ as $\lambda\rightarrow\infty$. 
%
In other words, 
\begin{align}
	\label{eq:weak1}
		\lim_{\lambda\rightarrow\infty}\mathds{E}\left[h(\mathcal{F}^\lambda)\right] = \mathds{E}\left[h(\mathcal{B})\right]
\end{align}
whenever $h:\mathcal{C}([0,\infty))\mapsto\mathds{R}$ is a bounded Borel-measurable functional continuous with respect to the topology of uniform convergence on compact intervals. Weak convergence implies that the family of probability laws induced by $\{\mathcal{F}^\lambda\}_{\lambda>0}$ is tight, and that, for any $0\le t_1 < t_2 <\dots <t_n <\infty$, 
	\begin{align*} 
		\lim_{\lambda \rightarrow \infty} (F^\lambda(t_1),F^\lambda(t_2),\dots,F^\lambda(t_n)) \buildrel{d} \over{=} (B(t_1), B(t_2),\dots, B(t_n)).
	\end{align*} 
These two properties are also sufficient conditions for (\ref{eq:weak1}) to hold \cite{Billingsley:1999uc}. As weak convergence is a statement regarding probability laws, the stochastic processes involved do not need to be defined on a common probability space. 

An alternative definition of the flip-flop process $\mathcal{F}^\lambda$ is as follows. For $t > 0$, let 
\[N^\lambda(t) =\#\{s\in(0,t]: \varphi^\lambda(s^-) \neq \varphi^\lambda(s)\},\] 
and $N^\lambda(0)= 0$. Then, $\{N^\lambda (t)\}_{t\ge 0}$ is the Poisson process of intensity $\lambda$ which counts the jumps of $\{\varphi^\lambda(t)\}_{t\ge 0}$, and we can rewrite (\ref{eq:flipflop1}) as
\begin{align}
	\label{eq:telegraph1}
	F^\lambda(t) = \sqrt{\lambda}\int_0^t (-1)^{N^\lambda (s)}\dd s,\quad t\ge 0.
\end{align}
The process $\mathcal{F}^\lambda$ defined as in (\ref{eq:telegraph1}) was first considered in \cite{goldstein1951, kac1974}, 
	%
	%
where a link between its transition probabilities and the telegraph equation was developed. In this context, $\mathcal{F}^\lambda$ became known as a \emph{telegraph process} or \emph{uniform transport process}, of which the weak convergence to $\mathcal{B}$ was proved in \cite{pinsky1968differential} and \cite{watanabe1968approximation}. 

Later on, it was proved in \cite{griego1971almost} that such a convergence also holds in a pathwise sense. More precisely, the authors showed that there exists a common probability space in which the family of flip-flop (or uniform transport) processes $\{\mathcal{F}^\lambda\}_{\lambda>0}$ and a standard Brownian motion $\mathcal{B}$ are defined such that for any $T>0$
	\begin{align}
		\label{eq:strong1}
			\lim_{\lambda\rightarrow \infty}\sup_{0\le t\le T}\left| F^\lambda(t) - B(t) \right| = 0\quad \mbox{almost surely.}
	\end{align}
Whenever (\ref{eq:strong1}) holds, we say that $\mathcal{F}^\lambda$ converges \emph{strongly} to $\mathcal{B}$ as $\lambda\rightarrow\infty$. By applying the Bounded Convergence Theorem to (\ref{eq:weak1}), we trivially get that strong convergence implies weak convergence.
Strong convergence results also lead to stronger
approximations for diffusions and for solutions to stochastic differential equations (e.g. in \cite{gorostiza1979strong} and \cite{gorostiza1980rate2}, respectively). In \cite{gorostiza1980rate}, the rate of strong convergence of $\mathcal{F}^\lambda$ to $\mathcal{B}$ was computed. The key step in \cite{griego1971almost, gorostiza1980rate} consisted in embedding certain values of $\mathcal{F}^\lambda$ into $\mathcal{B}$ using the Skorokhod embedding theorem. 
In recent years, the study of flip-flop processes was generalised into different directions, most of which are based on the following. Consider a process $(\mathcal{R}, \mathcal{J}) = \{(R(t), J(t))\}_{t\ge 0}$ where the \emph{phase} process $\mathcal{J}$ is a Markov jump process on a finite state space $\mathcal{S}$, initial distribution $\boldsymbol{p}$, and intensity matrix $Q$, and the \emph{level} process $\mathcal{R}$ is defined by
%
%
\begin{align}
	\label{eq:Rt1}
	R(t) = \int_0^\infty \mu_{J(s)}\dd s + \int_0^\infty \sigma_{J(s)}\dd B(s),\quad t\ge 0,
\end{align}
with $\mu_i\in\mathds{R}$ and $\sigma_i \ge 0$ for $i\in\mathcal{S}$. 
If $\sigma_i = 0$ for all $i\in\mathcal{S}$, the process $(\mathcal{R}, \mathcal{J})$ is known as a \emph{stochastic fluid process} (SFP). If $\sigma_i>0$ for all $i\in\mathcal{S}$, then $(\mathcal{R}, \mathcal{J})$ is called a \emph{Markov modulated Brownian motion} (MMBM). In \cite{latouche2015morphing}, it is shown that there exists a family of SFPs that converges \emph{weakly} to any given MMBM. This result was later used to study MMBM with two boundaries in \cite{latouche2015fluid}, \cite{latouche2016feedback} and \cite{ahn2017time}, Markov-modulated sticky Brownian motion in \cite{latouche2017slowing}, and MMBM with temporary change of regime at zero in \cite{latouche2018markov}. 

In this paper, we construct a sequence of stochastic fluid processes which converges \emph{strongly} to an MMBM of any given parameters. More specifically, we prove the following result. 
%
%
\begin{theorem}
	\label{th:strongConvMMBM}
For any given $\boldsymbol{p}$, $Q$, $\{\mu_i\}_{i\in\mathcal{S}}$ and $\{\sigma_i > 0\}_{i \in\mathcal{S}}$, there exists a probability space $(\Omega, \mathscr{F}, \mathds{P})$ on which live an MMBM $(\mathcal{R}, \mathcal{J})=\{(R(t), J(t))\}_{t\ge 0}$ defined as in (\ref{eq:Rt1}) and a sequence of stochastic fluid models $\{(\mathcal{R}^n, \mathcal{J}^n)\}_{n \geq 0} =\{(R^n(t), J^n(t))\}_{t\ge 0}$, where $\mathcal{J}^n$ has the state space $\{+,-\}\times\mathcal{S}$, such that for all $T \geq 0$ 
\begin{align} 
	\lim_{n\rightarrow \infty} \sup_{0\le s\le T} \left| R(s) - R^n(s) \right| & = 0 \quad \mbox{a.s.}, \\ 
	\lim_{n\rightarrow\infty}\pi_2(J^n(T)) & = J(T) \quad \mbox{a.s.}, 
\end{align} 
where $\pi_2:\{+,-\}\times\mathcal{S}\mapsto \mathcal{S}$ denotes the second-coordinate projection.
\end{theorem}
%
%
In fact, Theorem \ref{th:strongConvMMBM} is a consequence of the following result which concerns the rate of the strong convergence of $\{(\mathcal{R}^n, \mathcal{J}^n)\}_{n \geq 0}$ to $(\mathcal{R}, \mathcal{J})$.

\begin{theorem}
	\label{th:rate}
Fix $T\in[0,1)$. In the probability space $(\Omega, \mathscr{F}, \mathds{P})$ of Theorem \ref{th:strongConvMMBM}, 
\begin{itemize} 
	\item[(i)] for each $q>0$ there exists a constant $\alpha=\alpha(q)>0$ such that
\begin{align}
	\label{eq:rate2}
\mathds{P}\left(\sup_{0 \le s\le T} \left|R(s) - R^n(s)\right| > \alpha \varepsilon_n\right) = o(n^{-q})\quad\mbox{as } n\rightarrow\infty, 
\end{align} 
with $\varepsilon_n := n^{-1/2} \log(n)$, where $o(g(n))$ for $g:\mathds{N} \mapsto \mathds{R}_+$ denotes a function $f:\mathds{N} \mapsto \mathds{R}$ such that $\lim_{n\rightarrow\infty} f(n)/g(n)= 0$;

	\item[(ii)] furthermore, the process $\{\pi_2(J^n(t))\}_{t\ge 0}$ converges in an a.s. local uniform sense to $\{J(t)\}_{t\ge 0}$; that is,
	\begin{equation}
		\label{eq:rateJ1}
		\lim_{\rho\downarrow 0}\left[\limsup_{n \rightarrow \infty} \left( \sup_{s\in(T-\rho,T+\rho)} d \left(\pi_2(J^n(s)), J(s)\right)\right)\right] =0\quad\mbox{a.s.,}
	\end{equation}
where $d(\cdot,\cdot)$ denotes the discrete metric in $\mathcal{S}$.
\end{itemize} 
\end{theorem}  
The case $T\in[0,1)$ of Theorem \ref{th:strongConvMMBM} is a consequence of Theorem \ref{th:rate} and the Borel-Cantelli lemma, with the case $T\ge 1$ following by elementary time-scaling arguments.
%

\begin{Remark} The proof of Theorem \ref{th:rate} is inspired by the work of \cite{gorostiza1980rate}, where we replace the use of the Skorokhod embedding theorem with a Poissionian observations argument. Our approach yields tighter and simpler bounds, which ultimately enables us to obtain a faster rate of convergence than the one of \cite{gorostiza1980rate} (which was proportional to $n^{-1/2}(\log(n))^{5/2}$) when reduced to the case of the standard Brownian motion.
\end{Remark} 

This paper is structured as follows. In Section \ref{sec:ProofStrongConvMMBM} we construct $(\Omega, \mathscr{F}, \mathds{P})$ and describe the distributional characteristics of each stochastic fluid process $(\mathcal{R}^n, \mathcal{J}^n)$, for $n\ge 0$. We compute in Section \ref{sec:rate} the rate of convergence of $\mathcal{R}^n$ to $\mathcal{R}$, from which the proof of Theorem \ref{th:rate}, and thus that of Theorem \ref{th:strongConvMMBM}, follows. Finally, in Section \ref{sec:applications} we develop some implications of Theorem \ref{th:strongConvMMBM} regarding the downcrossing probabilities of $\mathcal{R}$ and $\mathcal{R}^n$; in particular, we exhibit a new link between the solutions of certain Riccati and quadratic matrix equations.
%


\section{Construction of $\{(\mathcal{R}^n, \mathcal{J}^n)\}_{n \geq 0}$.}
	\label{sec:ProofStrongConvMMBM}
First, we construct the probability space suitable to prove Theorems \ref{th:strongConvMMBM} and \ref{th:rate}. Fix $\boldsymbol{p}$, $Q$, $\{\mu_i\}_{i \in\mathcal{S}}$ and $\{\sigma_i >0\}_{i\in\mathcal{S}}$ of Theorem \ref{th:strongConvMMBM}. Let $\lambda_0=2\max_{i\in\mathcal{S}}|Q_{ii}|$,
%
%
and consider a sequence $\{\lambda_n\}_{n\ge 1}$ such that $\lambda_n\ge \lambda_{n-1}$ for $n\ge 1$ and $\lim_{n\rightarrow\infty}\lambda_n = \infty$. Let $(\Omega, \mathscr{F}, \mathds{P})$ be a probability space that supports:
\begin{itemize}
	\item a standard Brownian motion $\mathcal{B}=\{B(t)\}_{t\ge 0}$,
	\item a Poisson process $\mathcal{M}^0=\{M^0(t)\}_{t\ge 0}$ of rate $\lambda_0/2$,
	\item a sequence of Poisson processes $\{\mathcal{\widetilde{M}}^n\}_{n\ge 1}$, where $\mathcal{\widetilde{M}}^n$ has rate $(\lambda_{n}-\lambda_{n-1})/2$,
	\item a discrete-time Markov chain $\mathcal{X}^0=\{X^0(k)\}_{k\ge 0}$ with state space $\mathcal{S}$, initial distribution $\boldsymbol{p}$, and transition probability matrix $P_0:=I + (\lambda_0/2)^{-1}Q$,
\end{itemize}
%
%
%
with $\mathcal{B}$, $\mathcal{M}^0$, $\{\mathcal{\widetilde{M}}^n\}_{n\ge 1}$, and $\mathcal{X}^0$ being independent of each other. All the elements stated in Theorem \ref{th:strongConvMMBM} and of the whole manuscript will be constructed in $(\Omega, \mathscr{F}, \mathds{P})$. To construct $(\mathcal{R},\mathcal{J})$ on $(\Omega, \mathscr{F}, \mathds{P})$, let 
\begin{equation}
	\label{eq:Jtunif1} 
	J(t)=X^0(M^0(t)),\quad t\ge 0.
\end{equation}
The uniformization method 
implies that $\mathcal{J}=\{J(t)\}_{t\ge 0}$ is a Markov jump process with initial distribution $\boldsymbol{p}$ and intensity matrix $Q$. Let $\mathcal{R}=\{R(t)\}_{t\ge 0}$ be defined as on (\ref{eq:Rt1}), so that $(\mathcal{R},\mathcal{J})$ corresponds to a Markov-modulated Brownian motion. 

Next, for each $n \geq 0$, we construct the process $(\mathcal{R}^n,\mathcal{J}^n)$ as follows. Define the arrival process $\mathcal{M}^{n} = \{M^n(t)\}_{t \geq 0}$ to be the superposition of $\{\mathcal{M}^0, \mathcal{\widetilde{M}}^{1}, \mathcal{\widetilde{M}}^{2},\dots, \mathcal{\widetilde{M}}^{n}\}$. Then, $\mathcal{M}^n$ is itself a Poisson process of intensity 
\begin{align*} 
	\lambda_0/2 + \sum_{\ell=1}^n(\lambda_\ell - \lambda_{\ell-1})/2 = \lambda_{n}/2,
\end{align*}  
and its arrival epochs form a subset of the arrival epochs of $\mathcal{M}^{n + m}$ for any $m\ge 0$. In other words, $\{\mathcal{M}^n\}_{n\ge 0}$ is a sequence of Poisson process with nested time epochs whose new arrivals, as $n$ increases, are created independently of the existing ones. Let us emphasize that choosing to have Poissonan observations with rates $\lambda_{n}/2$ allows a direct comparison of our construction with the models of \cite{gorostiza1980rate} and of \cite{ramaswami2013fluid} in the special case of flip-flop approximations to a standard Brownian motion.

Intuitively, our aim is to construct $(\mathcal{R}^n,\mathcal{J}^n)$ in such a way that $\mathcal{R}^n$ visits the levels of $\mathcal{R}$ inspected at the arrival epochs of the Poisson process $\mathcal{M}^n$. To that end, we employ the well-known Wiener-Hopf factorisation for the Brownian motion with drift; see \cite[Corollary 2.4.10]{bladt2017matrix} for a proof.
%
%
\begin{theorem}[Wiener-Hopf factorisation for BM]
	\label{th:WHBM1}
 Let $\{W_t\}_{t\ge 0}$ be a Brownian motion with variance $\sigma^2>0$, drift $\mu$, and initial point $W_0=0$. Let $S$ be a stopping time and let $T \sim \mbox{exp}(\beta)$, independent of $\{W_t\}_{t\ge 0}$. Then, $W_S - \min_{0\le t\le T} W_{S + t}$ and $W_{S+T} - \min_{0\le t\le T} W_{S + t}$ are independent and exponentially distributed with rates
\[ 
	\omega = \sqrt{\frac{\mu^2}{\sigma^4} + \frac{2\beta}{\sigma^2}} + \frac{\mu}{\sigma^2}\quad\mbox{and}\quad\eta = \sqrt{\frac{\mu^2}{\sigma^4} + \frac{2\beta}{\sigma^2}} - \frac{\mu}{\sigma^2}, \quad\mbox{respectively.}
\]
\end{theorem}

Theorem \ref{th:WHBM1} implies that, restricted to an exponentially distributed time interval, we can track both the value of the minimum over this period and that at the right endpoint of a Brownian motion with drift. 
Let $\{T_k^n\}_{k\ge 1}$ be the interarrival times of the process $\mathcal{M}^n$, and define $\theta_0^n := 0$, 
\begin{align}\label{eq:thetakn1}
	\theta_k^n:= \sum_{j=1}^k T^n_j, \quad k \geq 0;
\end{align}
thus $\{\theta_k^n\}_{n \geq 0}$ are the arrival epochs of $\mathcal{M}^n$. See Figure~\ref{fig:MnJt} for an illustration. 

\begin{figure}[!ht]
\begin{center} 
\includegraphics[scale=0.7]{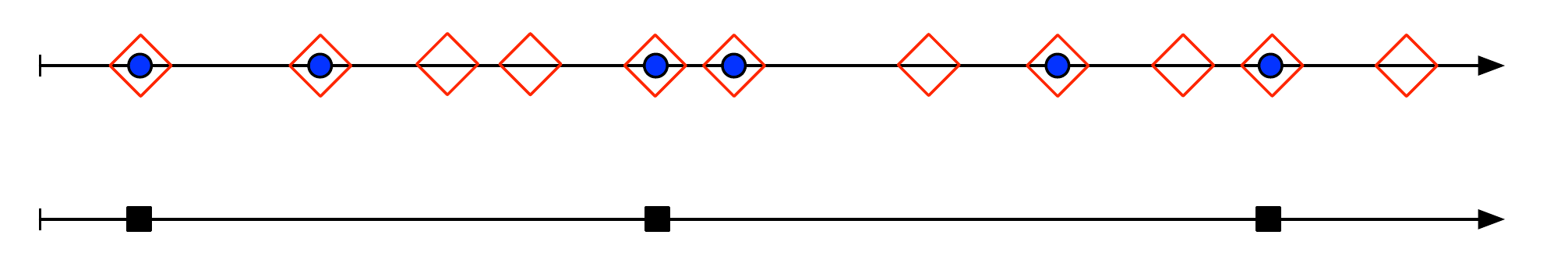} 
\put(-399, 2){\makebox(0,0)[l]{0}}
\put(-399, 41){\makebox(0,0)[l]{0}}
\put(-374, 36){\makebox(0,0)[l]{$\theta^{0}_1$}}
\put(-327, 36){\makebox(0,0)[l]{$\theta^{0}_2$}}
\put(-239, 36){\makebox(0,0)[l]{$\theta^{0}_3$}}
\put(-220, 36){\makebox(0,0)[l]{$\theta^{0}_4$}}
\put(-135, 36){\makebox(0,0)[l]{$\theta^{0}_5$}}
\put(-79, 36){\makebox(0,0)[l]{$\theta^{0}_6$}}
\put(-374, 66){\makebox(0,0)[l]{$\theta^{n}_1$}}
\put(-327, 66){\makebox(0,0)[l]{$\theta^{n}_2$}}
\put(-292, 66){\makebox(0,0)[l]{$\theta^{n}_3$}}
\put(-272, 66){\makebox(0,0)[l]{$\theta^{n}_4$}}
\put(-239, 66){\makebox(0,0)[l]{$\theta^{n}_5$}}
\put(-220, 66){\makebox(0,0)[l]{$\theta^{n}_6$}}
\put(-168, 66){\makebox(0,0)[l]{$\theta^{n}_7$}}
\put(-135, 66){\makebox(0,0)[l]{$\theta^{n}_8$}}
\put(-105, 66){\makebox(0,0)[l]{$\theta^{n}_9$}}
\put(-79, 66){\makebox(0,0)[l]{$\theta^{n}_{10}$}}
\put(-45, 66){\makebox(0,0)[l]{$\theta^{n}_{11}$}}
\caption{Blue dots correspond to arrivals $\{\theta^{0}_k\}_{k \geq 0}$ of $\mathcal{M}^0$, red diamonds the arrivals $\{\theta^n_k\}_{k \geq 0}$ of $\mathcal{M}^n$, black squares the jump epochs of $\mathcal{J}$. As $\mathcal{J}$ is given by \eqref{eq:Jtunif1}, its jump epochs form a subset of the arrival times of $\mathcal{M}^0$.} 
\label{fig:MnJt} 
\end{center} 
\end{figure}
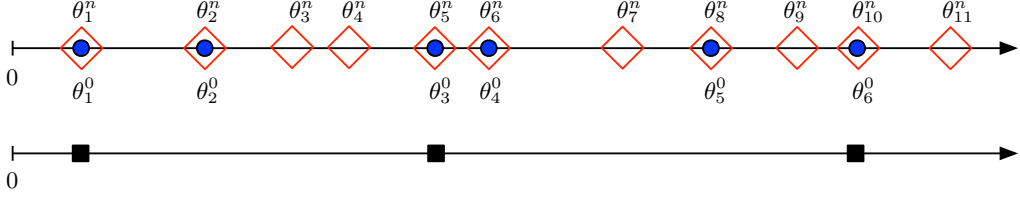
%

%
As $\{\theta_k^n\}_{n \geq 0}$ contain all the arrival epochs of $\mathcal{M}^0$, Equation (\ref{eq:Jtunif1}) implies that $\{J(t)\}_{t\ge 0}$ remains constant on each interval $[\theta_k^n, \theta_{k+1}^n)$, $k\ge 0$. Consequently, given $J({\theta_k^n})=i$ on $[\theta_k^n, \theta_{k+1}^n)$, $\{R(t)\}_{t\ge 0}$ behaves like a Brownian motion with drift $\mu_i$ and variance $\sigma^2_i$. Thus, by sequentially using the Wiener-Hopf factorisation between arrival epochs of $\mathcal{M}^n$, we can keep track of $\{ R(\theta_k^n)\}_{k\ge 0}$ and of $\{\min_{\theta_k^n\le t\le \theta_{k+1}^n} R(t)\}_{k\ge 0}=\{\min_{0\le t\le T_{k+1}^n} R(\theta_k^n + t)\}_{k\ge 0}$ in a simple manner, which we explain in detail next. 
%

For each $k\ge 0$, define the random variables%
%
%
%
\begin{align*} 
X^n(k) & := J(\theta_k^n), \\
L^{n}_{k+1} & :=  R(\theta_{k}^n) - \min_{0\le t\le T_{k+1}^n} R(\theta_k^n + t),\\
H^{n}_{k+1} & :=  R(\theta_{k+1}^n) - \min_{0\le t\le T_{k+1}^n} R(\theta_k^n + t). 
\end{align*} 
By Theorem \ref{th:reverseUnif} in the Appendix, $\mathcal{X}^n=\{X^n(k)\}_{k\ge 0}$ is a discrete-time Markov chain with transition probability matrix $P_n:= I + (\lambda_n/2)^{-1}Q$. The strong Markov property of $\{(R(t), J(t))\}_{t\ge 0}$ and Theorem \ref{th:WHBM1} imply that, conditioned on $\mathcal{X}^n$, $\{L^{n}_{k+1}\}_{k\ge 0}$ is a collection of independent random variables. More specifically, given $X^n_k=i$, $L^{n}_{k+1}$ is exponentially distributed with rate
\[
	\omega_i^n := \sqrt{\frac{\mu^2_i}{\sigma^4_i} + \frac{\lambda_n}{\sigma^2_i}} + \frac{\mu_i}{\sigma^2_i}.
\]
Similarly, $\{H^{n}_{k+1}\}_{k\ge 0}$ is a collection of conditionally 
independent random variables for which, given $X^n_k=i$, $H^{n}_{k+1}$ is exponentially distributed with rate
\begin{align*} 
	\eta_i^n := \sqrt{\frac{\mu^2_i}{\sigma^4_i} + \frac{\lambda_n}{\sigma^2_i}} - \frac{\mu_i}{\sigma^2_i}.
\end{align*} 
Moreover, $\{L^{n}_{k+1}\}_{k\ge 0}$ is conditionally independent of $\{H^{n}_{k+1}\}_{k\ge 0}$. Note that $\{L^{n}_{k+1}\}_{k\ge 0}$ and $\{H^{n}_{k+1}\}_{k\ge 0}$ completely describe $\{ R(\theta_k^n)\}_{k\ge 0}$ and $\{\min_{\theta_k^n\le t\le \theta_{k+1}^n} R(t)\}_{k\ge 0}$, in the sense that for all $k\ge 0$,
\begin{align}
	\label{eq:sumH1}
R(\theta_k^n)& = \sum_{j=1}^k \left( -L^{n}_{j} +H^{n}_{j}\right), \quad \mbox{and}\\
	\label{eq:sumH2}
	\min_{\theta_k^n\le t\le \theta_{k+1}^n} R(t)& = \sum_{j=1}^k \left(-L^{n}_{j} +H^{n}_{j}\right) - L^{n}_{k+1}.
\end{align}
For all $k\ge 1$, if $X^n(k-1) = i$, define  
\begin{align*}
\widehat{L}^{n}_{k} & := \lambda_n^{-1}\omega_i^n L^{n}_{k} \quad \mbox{ and } \quad \widehat{H}^{n}_{k}  := \lambda_n^{-1}\eta_i^n H^{n}_{k}. 
\end{align*}
Then, the collections $\{\widehat{L}^{n}_{k}\}_{k\ge 1}$ and  $\{\widehat{H}^{n}_{k}\}_{k\ge 1}$ are i.i.d. random variables exponentially distributed with parameter $\lambda_n$. Let $\chi_n^0 := 0$, and define for all $k\ge 1$ 
\begin{align*} 
	\quad\chi_k^n:=\sum_{j=1}^k \left(\widehat{L}^{n}_{j} + \widehat{H}^{n}_{j}\right).
\end{align*} 
Let $\mathcal{J}^n = \{J^n(t)\}_{t\ge 0}$ be the process with state space $\{+,-\}\times\mathcal{S}$ defined by
\begin{align*} 
	J^n(t) =\left\{\begin{array}{cl}(-,i) 
	& \mbox{if } t\in[\chi^n_k,\chi^n_k + \widehat{L}^{n}_{k+1})\mbox{ for some $k\ge 0$ and }X_k^n = i,\\
\vspace*{-0.2cm} \\
(+,i) & \mbox{if } t\in[\chi^n_k + \widehat{L}^{n}_{k+1}, \chi^n_{k+1})\mbox{ for some $k\ge 0$ and }X_k^n = i.
	\end{array}\right.
\end{align*} 
	Figure~\ref{fig:SandChi} shows a sample path of $\mathcal{J}^n$ with $\mathcal{S} = \{1, 2, 3\}$.
\begin{figure}[h]
	\centering
	\includegraphics[scale=0.5]{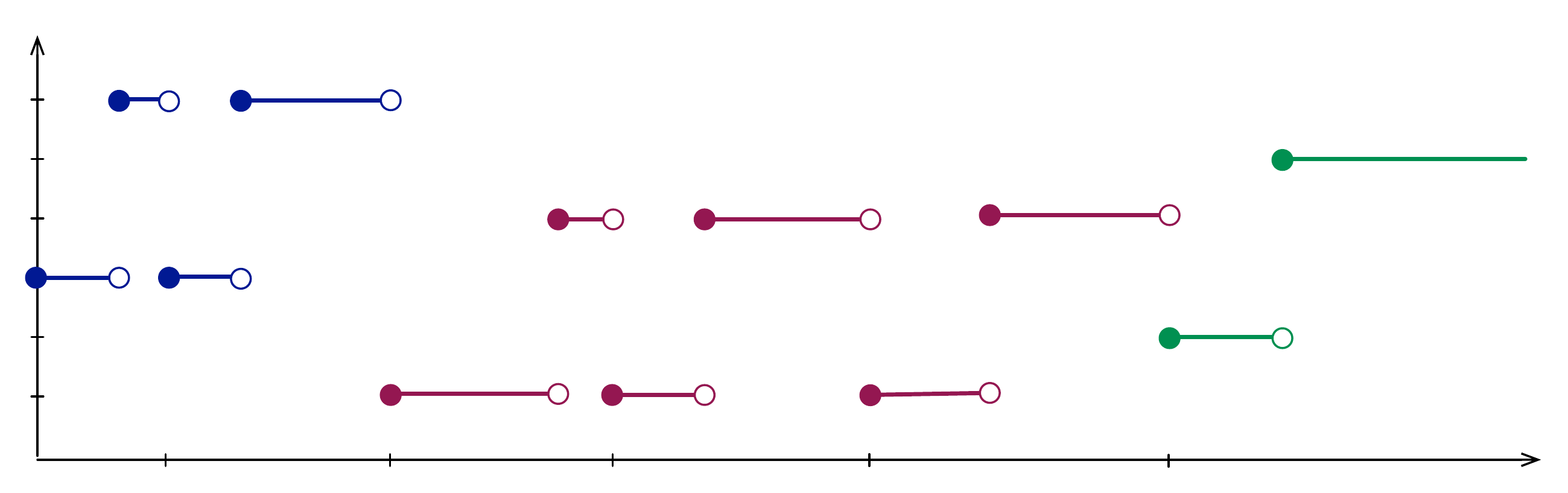}
		 \put(-382,117){\makebox(0,0)[l]{$J^n(t)$}}
		 \put(-3,8){\makebox(0,0)[l]{$t$}}
		 \put(-372,0){\makebox(0,0)[l]{$0$}}
		  \put(-340,0){\makebox(0,0)[l]{$\chi^n_1$}}
		  \put(-285,0){\makebox(0,0)[l]{$\chi^n_2$}}
		  \put(-233,0){\makebox(0,0)[l]{$\chi^n_3$}}
	           \put(-172,0){\makebox(0,0)[l]{$\chi^n_4$}}
		  \put(-100,0){\makebox(0,0)[l]{$\chi^n_5$}}
		  \put(-400,22){\makebox(0,0)[l]{$(-,3)$}} 
		  \put(-400,37){\makebox(0,0)[l]{$(-,2)$}} 
		\put(-400,51){\makebox(0,0)[l]{$(-,1)$}} 
		\put(-400,65){\makebox(0,0)[l]{$(+,3)$}} 
		 \put(-400,79){\makebox(0,0)[l]{$(+,2)$}} 
		 \put(-400,94){\makebox(0,0)[l]{$(+,1)$}} 
\caption{A sample path of $\{J^n(t)\}$ with $\mathcal{S} = \{1,2, 3\}$.	\label{fig:SandChi} }
\end{figure}
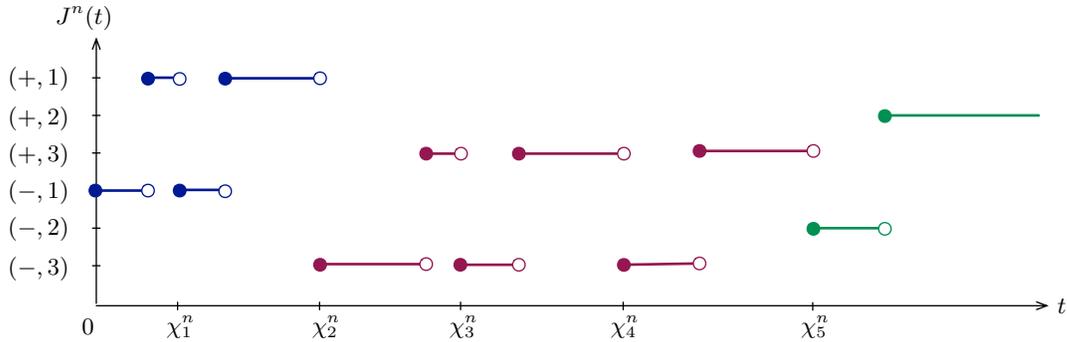

The process $\mathcal{J}^n$ jumps alternately between $\{-\}\times\mathcal{S}$ and $\{+\}\times\mathcal{S}$ with intensity given by $\lambda_n$; furthermore, changes in its second coordinate, which occur according to $P_n$, are only possible at jumps instants from $\{+\}\times\mathcal{S}$ to $\{-\}\times\mathcal{S}$. Thus, $\mathcal{J}^n$ is a Markov jump process with state-space $\{+,-\}\times\mathcal{S}$ (ordered lexicographically), initial distribution $(\bm{0},\bm{p})$ and intensity matrix given by
\[
	\left[\begin{array}{rr} 
	-\lambda_n I & \lambda_n P_n\\
\lambda_n I & -\lambda_n I	
\end{array}\right] = \left[\begin{array}{rr} -\lambda_n I & 2 Q + \lambda_n I\\
\lambda_n I& -\lambda_n I					
	\end{array}\right]. 
\]
%
%
%
Note that the sequence of states in $\mathcal{S}$ visited by $\pi_2(\mathcal{J}^n)$ coincides with that of $\mathcal{J}$, or more precisely, 
\begin{equation}\label{eq:sameJandJn}\pi_2({J}^n(\chi^n_k))=J(\theta^n_k)\quad\mbox{for all}\quad k\ge 0.\end{equation}
Also notice the jumps of $\pi_2(\mathcal{J}^n)$ can occur only at $\{\chi^n_k\}_{k\ge 0}$ while the jumps of $\mathcal{J}$ can occur only at $\{\theta^n_k\}_{k\ge 0}$. In general, $\{\chi^n_k\}_{k\ge 0}\neq\{\theta^n_k\}_{k\ge 0}$; nevertheless,
\begin{align*}
\mathds{E}\left[\theta^n_k\right]& = \mathds{E}\left[\sum_{j=1}^k T^n_j\right] = k \mathds{E}\left[T^n_1\right] = \frac{2k}{\lambda_n},\\
\mathds{E}\left[\chi^n_k\right]&=\mathds{E}\left[\sum_{j=1}^k \left(\widehat{L}^{n}_{j} + \widehat{H}^{n}_{j}\right)\right] = k \mathds{E}\left[\widehat{L}^{n}_{1} + \widehat{H}^{n}_{1}\right] = \frac{2k}{\lambda_n}.
\end{align*}
In words, the average jump times of $\pi_2(\mathcal{J}^n)$ coincide with the average jump times of $\mathcal{J}$, so that the process $\pi_2(\mathcal{J}^n)$ is indeed {\it similar} to $\mathcal{J}$. A more precise and stronger version of this statement is proven in Section \ref{sec:rate}.

In order to construct $\mathcal{R}^n$, define $r^n:\{+,-\}\times\mathcal{S} \mapsto \mathds{R}$ by
\begin{align*}
	r^n(+,i) & := \lambda_n/\omega_i^n, \quad r^n(-,i)  := -\lambda_n/\eta_i^n.
\end{align*}
Let
\[
	R^n(t) := \int_0^t r^n(J^n(s))\dd s,\quad t\ge 0.
\]
The pair $(\mathcal{R}^n, \mathcal{J}^n)$ is indeed a stochastic fluid process. Moreover, from the construction of $\mathcal{J}^n$, (\ref{eq:sumH1}) and (\ref{eq:sumH2}), it follows that for all $k \ge 0$
\begin{align}
	\label{eq:RnS1} 
	R^n\left(\chi_k^n\right) & = \sum_{j=1}^k\left(-L^{n}_{j} + H^{n}_{j}\right) = R(\theta_k^n), \\ 
	\label{eq:RnS2}  
	R^n\left(\chi_k^n + \widehat{L}^{n}_{k+1}\right) & = \sum_{j=1}^k \left(- L^{n}_{j} + H^{n}_{j}\right) - L^{n}_{k+1}=\min_{\theta_k^n\le t\le \theta_{k+1}^n} R(t).
\end{align}
This implies that the values at the inflection points of the level process $\mathcal{R}^n$ coincide with the values of $\{ R(\theta_k^n)\}_{k\ge 0}$ and $\{\min_{\theta_k^n\le t\le \theta_{k+1}^n} R(t)\}_{k\ge 0}$. In conclusion, the values of $\mathcal{R}$ at the arrival epochs of $\mathcal{M}^n$, and the minimum level attained between them, are embedded in $\mathcal{R}^n$. Figure \ref{fig:flip-flop} illustrates the construction of the stochastic fluid process $(\mathcal{R}^n, \mathcal{J}^n)$ corresponding to the Markov-modulated Brownian motion $(\mathcal{R}, \mathcal{J})$.

\begin{figure}[!ht]
	\begin{center} 
	\includegraphics[scale=0.4]{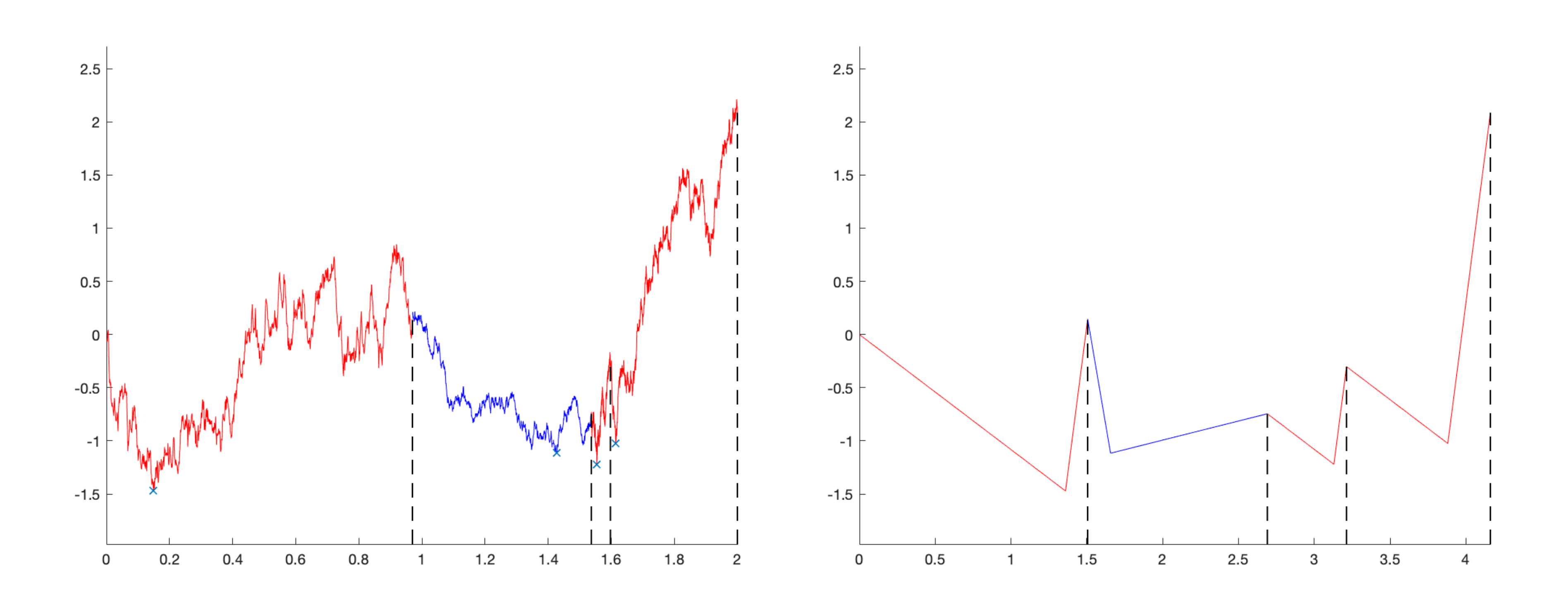}
	 \put(-138,7){\makebox(0,0)[l]{$\chi^n_1$}}
	\put(-87,7){\makebox(0,0)[l]{$\chi^n_2$}}
	\put(-63,7){\makebox(0,0)[l]{$\chi^n_3$}}
	\put(-23,7){\makebox(0,0)[l]{$\chi^n_4$}}
	\put(-328,7){\makebox(0,0)[l]{$\theta^n_1$}}
	\put(-276,7){\makebox(0,0)[l]{$\theta^n_2$}}
	\put(-268,7){\makebox(0,0)[l]{$\theta^n_3$}}
	\put(-236,7){\makebox(0,0)[l]{$\theta^n_4$}}
	\put(-415,164){\makebox(0,0)[l]{$R(t)$}}
	\put(-204,164){\makebox(0,0)[l]{$R^n(t)$}}
	\put(-221,20){\makebox(0,0)[l]{$t$}}
	\put(-11,20){\makebox(0,0)[l]{$t$}}
	\end{center} 
\caption{\textbf{(Left)} A sample path of an MMBM $(\mathcal{R},\mathcal{J})$ with $\mathcal{J}$ being on $\mathcal{S} = \{1, 2\}$: arrivals corresponding to $\mathcal{M}^n$ occur at $\{\theta_i^n\}_{i \geq 0}$, and the minima of $\mathcal{R}$ attained between these arrivals are highlighted with blue crosses. When $J(t) = 1$ (red), $\mu_1 = 5$ and $\sigma_1^2 = 4$. When $J(t) = 2$ (blue), $\mu_2 = -2$ and $\sigma_2^2 = 1$. \textbf{(Right)} An associated sample path of a stochastic fluid process $(\mathcal{R}^n,\mathcal{J}^n)$: jumps from $\{+\}\times\mathcal{S}$ to $\{-\}\times\mathcal{S}$ occur at $\{\chi_i^n\}_{i \geq 1}$. The values of $R^n(\chi_i^n)$ match with those of $R(\theta_i^n)$ for $i = 1, \ldots 4$, respectively.}
\label{fig:flip-flop}
\end{figure}
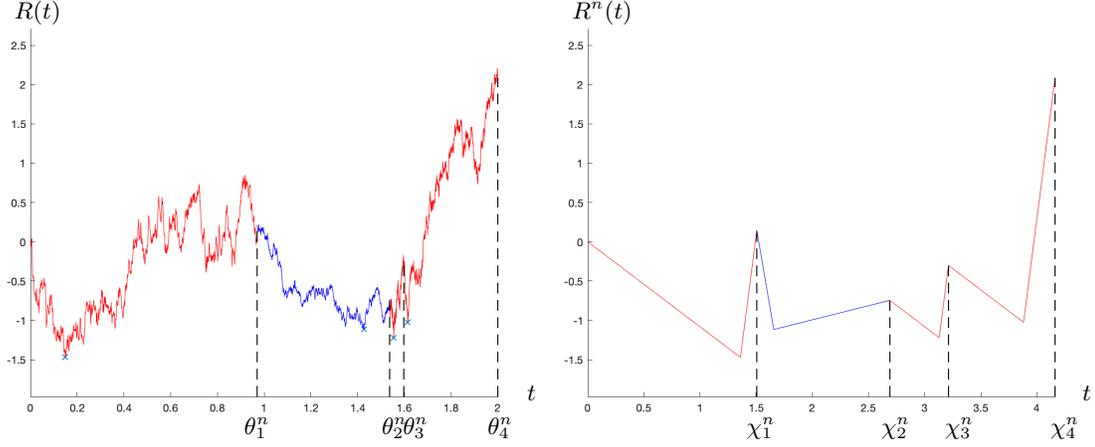
\section{Proof of Theorem \ref{th:rate}.}
	\label{sec:rate}
As $\lambda_n \rightarrow \infty$, the partitions induced by $\{\theta^n_k\}_{k\ge 0}$ and $\{\chi^n_k\}_{k\ge 0}$ become finer. Intuitively, this and (\ref{eq:RnS1})  indicate that $\mathcal{R}^n$ \emph{approximates} $\mathcal{R}$ as $n\rightarrow\infty$, which is stated more precisely in Theorems \ref{th:strongConvMMBM} and \ref{th:rate}. We devote this section to rigorously prove Theorem \ref{th:rate}, from which Theorem \ref{th:strongConvMMBM} follows as a corollary.

Let $\lambda_n=2n^2$; this makes our results and rates comparable to those of \cite{gorostiza1980rate} and related papers. Fix $T \in [0,1)$, $q>0$ and w.l.o.g. consider $n\ge 2$ throughout. 

\textbf{Proof of Part (i).} In order to prove (\ref{eq:rate2}), notice that
\begin{align}
	\label{eq:rateaux1}
\mathds{P}\left(\sup_{0 \le s\le T} |R(s) - R^n(s)| > \alpha \varepsilon_n\right)\le \mathds{P}(A^n) + \mathds{P}(\chi^n_{n^2} < T), 
\end{align}
where 
\begin{align*} 
	A^n := \left\{\sup_{0 \le s\le \chi^{n}_{n^2}} |R(s) - R^n(s)| > \alpha \varepsilon_n\right\} = \left\{\max_{1 \leq k \leq n^2} \sup_{\chi^{n}_{k-1} \le s \le \chi^{n}_{k}} |R(s) - R^n(s)| > \alpha \varepsilon_n\right\}, 
\end{align*} 
where $\alpha$ is a constant to be determined later. We now show that each of the quantities $P(A^n)$ and $P(\chi_{n^2}^n < T)$ are $o(n^{-q})$.  

The triangle inequality implies that
\begin{align*} 
	A^n \subseteq B_1^n\cup B_2^n\cup B_3^n\cup B_4^n,
\end{align*} 
where
\begin{align*}
B_1^n & := \left\{\max_{1 \le k\le n^2}\sup_{\chi_{{k - 1}}^n \le s\le \chi_{{k}}^n } |R(s) - R(\chi_k^n)| > \alpha \varepsilon_n/4\right\},\\
\displaybreak
B_2^n & := \left\{\max_{1 \le k\le n^2} |R(\chi_k^n) - R(k/n^2)| > \alpha \varepsilon_n/4\right\},\\
B_3^n & := \left\{\max_{1 \le k\le n^2} |R(k/n^2) - R^n(\chi_k^n)| > \alpha \varepsilon_n/4\right\},\\
B_4^n & := \left\{\max_{1 \le k\le n^2}\sup_{\chi_{{k - 1}}^n \le s\le \chi_{{k}}^n} |R^n(\chi_k^n) - R^n(s)| > \alpha \varepsilon_n/4\right\}.
\end{align*}
By (\ref{eq:RnS1}), $B_3^n$ can be rewritten as 
	\[
	B_3^n = \left\{\max_{1 \le k\le n^2} |R(\theta_k^n) - R(k/n^2)| > \alpha \varepsilon_n/4\right\},
	\]
	so that the events $B_1^n, B_2^n$ and $B_3^n$ concern only the process $\mathcal{R}$, not $\mathcal{R}^n$. 
	
	Let $\{\delta_n\}_{n\ge 0}$ be any positive and decreasing sequence. Making a further partition, we obtain 
\begin{align} 
	\label{eqn:Anfurther} 
A^n
%
%
 \subseteq \left(\left(B_1^n\cup B_2^n\cup B_3^n\right) \cap \left(  C^n_\chi \cup C^n_\theta\right)^c \right) \cup \left( C^n_\chi \cup C^n_\theta\right) \cup B_4^n,
\end{align} 
where
\begin{align*}
 C^n_\chi := \left\{\max_{1 \le k\le n^2} |\chi_k^n - k/n^2|>\delta_n\right\}, \quad 
  C^n_\theta := \left\{\max_{1 \le k\le n^2} |\theta_k^n - k/n^2|>\delta_n\right\}.
\end{align*}
On $(C^n_\chi\cup C^n_\theta)^c$, for all $k = 1, \ldots, n^2$ we have that 
\begin{align*} 
	\chi_k^n, \theta_k^n, \chi_{k + 1}^n \in [k/n^2 - \delta_n, (k + 1)/n^2 + \delta_n]. 
\end{align*}
Let $x_+:=\max\{x,0\}$ for $x\in\mathds{R}$. If there exists $a, b \in [k/n^2 - \delta_n, (k + 1)/n^2 + \delta_n]$ such that $|R(a_+) - R(b_+)| > \alpha \varepsilon_n/4$ , then by the triangle inequality either $|R([k/n^2 - \delta_n]_+) - R(a_+)| > \alpha \varepsilon_n / 8$, or $|R([k/n^2  - \delta_n]_+) - R(b_+)| > \alpha \varepsilon_n / 8$. 
 Therefore, 
\begin{align} 
	\label{eqn:whyneedDn} 
	 \left(B_1^n\cup B_2^n\cup B_3^n\right) \cap  \left(C^n_\chi \cup C^n_\theta\right)^c
	\subseteq D^n,
\end{align} 
where
\begin{align*} 
	D^n &:= \left\{\max_{1 \le k\le n^2}\sup_{a \in [k/n^2 - \delta_n, (k + 1)/n^2 + \delta_n]} \left|R\left((k/n^2  - \delta_n)_+\right) - R\left(a_+\right)\right| > \alpha \varepsilon_n/8\right\}\\
	&\;= \left\{\max_{1 \le k\le n^2}\sup_{s \in [0, n^{-2} + 2\delta_n]} \left|R\left((k/n^2  - \delta_n + s)_+\right) - R\left((k/n^2  - \delta_n)_+\right)\right| > \alpha \varepsilon_n/8\right\}. 
	%
\end{align*}
Thus, by \eqref{eqn:Anfurther} and \eqref{eqn:whyneedDn}, 
\begin{align}
	\label{eq:auxsum1}
\mathds{P}(A^n) \le \mathds{P}(D^n) + \mathds{P}(C^n_\chi ) + \mathds{P}(C^n_\theta)   + \mathds{P}(B^n_4).
\end{align}
 
In the following, we show that with an appropiate choice of $\alpha$ and $\{\delta_n\}_{n\ge 1}$, each summand in the RHS of (\ref{eq:auxsum1}) is an $o(n^{-q})$ function. For the remainder of the section, $K_j$, for $j \in \mathds{N}$, denote generic constants that are used to simplify bounds and are not dependent on $q$ or $n$.

\textbf{Bounding $P(B^n_4)$.} Let 
	$\omega_{\min}^n := \min\limits_{i\in\mathcal{S}}\omega_i^n,  \; \eta_{\min}^n := \min\limits_{i\in\mathcal{S}}\eta_i^n, \;\kappa_n :=\omega_{\min}^n\wedge\eta_{\min}^n. 
	$
Then, 
\begin{align}
\mathds{P}(B_4^n)&\le \sum_{1\le k\le n^2}\mathds{P}\left(\sup_{\chi_{{k - 1}}^n \le s\le \chi_{{k}}^n}  \left| R^n(\chi_k^n) - R^n(s)\right| > \alpha \varepsilon_n/4\right) \nonumber \\
&\le \sum_{1\le k\le n^2}\mathds{P}\left(H^{n}_k>\alpha \varepsilon_n/4\right) + \mathds{P}\left(L^{n}_k>\alpha \varepsilon_n/4\right) \nonumber\\
&\le n^2\left(e^{-\omega^n_{\min}\alpha \varepsilon_n/4} + e^{-\eta^n_{\min}\alpha \varepsilon_n/4}\right)\nonumber\\
&\le n^2 \left( 2e^{-\kappa_n(\alpha \varepsilon_n/4)}\right). \nonumber
\end{align}
By definition, 
\begin{align*} 
	\kappa_n := \min_{i\in\mathcal{S}}\left\{\left(\sqrt{\frac{\mu^2_i}{\sigma^4_i} + \frac{2n^2}{\sigma^2_i}} + \frac{\mu_i}{\sigma^2_i}\right)\wedge\left( \sqrt{\frac{\mu^2_i}{\sigma^4_i} + \frac{2n^2}{\sigma^2_i}} - \frac{\mu_i}{\sigma^2_i}\right)\right\}= O(n),
\end{align*} 
where the notation $O(g(n))$, for $g:\mathds{N}\mapsto \mathds{R}_+$, denotes a function $f:\mathds{N}\mapsto \mathds{R}$ such that $\limsup_{n \rightarrow \infty} |f(n)|/g(n) \le M$ for some $M\in\mathds{R}$. 
Then, there exists $n_1$ such that $\kappa_n >n^{1/2}$ for all $n\ge n_1$, and so for $n\ge n_1$
%
\begin{align*}
	\mathds{P}(B_4^n)&\le n^2 \left(2e^{-\kappa_n(\alpha/4)n^{-1/2}\log(n)}\right)  \le n^2 \left(2e^{-(\alpha/4)\log(n)}\right) =K_1n^{2-\alpha/4},
\end{align*}
Choose $\alpha$ to be larger than $\alpha_1:=8q+8$. Then, $\mathds{P}(B_4^n)$ is an $O(n^{-2q})$ function and thus, it is an $o(n^{-q})$ function. 

\textbf{Bounding $P(C^n_{\chi})$ and $P(C^n_{\theta})$.} Let $\{p_n\}_{n}$ be a sequence taking values in $\mathds{N}$. By Doob's $L_p$-maximal inequality, we have 
\begin{align}
	\label{eq:Ansigmabound1}
\mathds{P}(C^n_\chi )\le\frac{\mathds{E}[(\chi_{n^2}^n - 1)^{2p_n}]}{(\delta_n)^{2p_n}}.
\end{align}
Since $\chi_{n^2}^n$ is a convolution of $2n^2$ exponential r.v.s of rate $2n^2$, $\chi_{n^2}^n\sim\mbox{Erlang}(2n^2, 2n^2)$, so that (\ref{eq:Ansigmabound1}) and Lemma \ref{lem:cmomentErl} (in the Appendix) imply that
\begin{align}
\mathds{P}(C^n_\chi )&\le(\delta_n)^{-2p_n} \frac{(2p_n)!\sqrt{2n^2}}{(2n^2)^{2p_n}} \frac{\sqrt{2n^2}^{2p_n+1}-1}{\sqrt{2n^2}-1}\nonumber\\
& \le K_2 \frac{(\delta_n)^{-2p_n}(2p_n)! }{n^{2p_n-1}}\nonumber\\
& \le K_2 n\left(\frac{2p_n }{\delta_n n}\right)^{2p_n}.\label{eq:boundAnsigma2}
\end{align}

Similarly, since $\theta_{n^2}^n\sim \mbox{Erlang}(n^2, n^2)$, we have for $n \geq 2$
\begin{align}\label{eq:boundAntheta2}
\mathds{P}(C^n_\theta)\le K_3 n\left(\frac{2p_n }{\delta_n n}\right)^{2p_n}.
\end{align}
Set 
\begin{align}
	\label{eq:deltan}
		\delta_n := 2p_nn^{(q+1/2)/p_n - 1},\quad n\ge 1. 
\end{align}
With this choice of $\{\delta_n\}$, both (\ref{eq:boundAnsigma2}) and the RHS of (\ref{eq:boundAntheta2}) are proportional to $n^{-2q}$, so that $\mathds{P}(C^n_\chi )$ and $\mathds{P}(C^n_\theta)$ are $o(n^{-q})$ functions. 

\textbf{Bounding $P(D^n)$.} Set $p_n := \lfloor\log(n)\rfloor$; one can verify that with this choice of $\{p_n\}$, the sequence $\{\delta_n\}$ is a $O(n^{-1}\log(n))$ function. Define $\delta_n' := 2\delta_n + n^{-2}$ and let $n_2\ge n_1$ be such that $\alpha_1\varepsilon_n/8-\mu_{\max}\delta_n'>0$ for all $n\ge n_2$. Then, for any $\alpha\ge \alpha_1$, we have $\alpha\varepsilon_n/8-\mu_{\max}\delta_n'>0$ for all $n\ge n_2$. Thus,
\begin{align}
 \mathds{P}(D^n)
 & \leq {\sum_{k = 1}^{n^2} \mathds{P} \left(\sup_{0 \leq s \leq \delta_n'} \left| R\left((k/n^2 - \delta_n)_+\right) - R\left((k/n^2 - \delta_n + s)_+\right) \right| > \alpha \varepsilon_n/8\right)}\nonumber \\
%
& \leq { \sum_{k = 1}^{n^2} \sum_{i \in \mathcal{S}} \mathds{P}\left(\left.\sup_{0 \leq s \leq \delta_n'} \left| R\left((k/n^2 - \delta_n)_+\right) - R\left((k/n^2 - \delta_n + s)_+\right) \right| > \alpha \varepsilon_n/8 \; \right| \; G_{i, k, n} \right)},  \nonumber
%
\intertext{where $G_{i,k,n} := \{J\left(k/n^2 - \delta_n\right) = i\}$. By strong Markov property, we can rewrite the above RHS to obtain} 
%
 \mathds{P}(D^n) & = {n^2\sum_{i \in \mathcal{S}} \mathds{P}\left(\left.\sup_{0 \leq s \leq \delta_n'} \left| R(s) \right| > \alpha \varepsilon_n/8 \; \right| \; J(0) = i \right)} \nonumber \\ 
%
%
&\le 
n^2m \left[\frac{2}{\sqrt{2\pi}}
\frac{\sqrt{\sigma_{\max}  \delta_n'}}{\alpha \varepsilon_n / 8 - \mu_{\max} \delta_n'}
\exp\left(-\frac{(\alpha\varepsilon_n/8-\mu_{\max}\delta_n')^2}{2\sigma_{\max}\delta_n'}\right)
\right]\nonumber \\
&\le K_4n^2\exp\left(-\frac{(\alpha\varepsilon_n/8-\mu_{\max}\delta_n')^2}{2\sigma_{\max}\delta_n'}\right),\quad n\ge n_2, \label{eq:PDn2}
\end{align}
%
%
where the first inequality follows from Lemma \ref{lem:boundforMMBM} (in the Appendix).
%

Let $n_3\ge n_2$ be such that $\varepsilon_n\le 1$ and $\delta_n'\le 3\delta_n$ for all $n\ge n_3$.
Then
\begin{align}
\mathds{P}(D^n)&\le  K_4n^2 \exp\left(-\frac{(\alpha\varepsilon_n/8)^2 + (\mu_{\max}\delta_n')^2 - 2(\alpha\varepsilon_n/8)(\mu_{\max}\delta_n') }{2\sigma_{\max}\delta_n'}\right) \nonumber\\
& \le  K_4 n^2 \exp\left(-\frac{(\alpha\varepsilon_n/8)^2  - 2(\alpha\varepsilon_n/8)(\mu_{\max}\delta_n') }{2\sigma_{\max}\delta_n'}\right) \nonumber\\
& \le  K_4 n^2 \exp\left(-\frac{(\alpha\varepsilon_n/8)^2}{2\sigma_{\max}\delta_n'}  + \frac{\alpha\mu_{\max}}{8\sigma_{\max}}\right)\nonumber\\
%
 & \le  K_5 n^2 \exp\left(-\frac{(\alpha\varepsilon_n/8)^2}{6\sigma_{\max}\delta_n}\right)\nonumber\\
&= K_5 n^2 \exp\left(-\frac{\alpha^2n^{-1}(\log(n))^2}{K_6\lfloor\log(n)\rfloor n^{(q+1/2)/\lfloor\log(n)\rfloor -1}}\right)\nonumber \\
&  \le  K_5 n^2 \exp\left(-\frac{\alpha^2\log(n)}{K_6 n^{(q+1/2)/\lfloor\log(n)\rfloor}}\right),\quad n\ge n_3.\label{eq:boundDn3}
\end{align}
Let $\gamma(q)=\sup_{n\ge n_3} K_6 n^{(q+1/2)/\lfloor\log(n)\rfloor}$, which is finite since $n^{(q+1/2)/\lfloor\log(n)\rfloor}$ converges to $e^{q + 1/2}$. If $\alpha > \alpha_2:=\sqrt{(2q+2)\gamma(q)}$, by (\ref{eq:boundDn3}) we have
\begin{align*} 
\mathds{P}(D^n) & \le K_5 n^2 \exp\left(-\frac{(2q + 2)\gamma(q)\log(n)}{ K_6 n^{(q+1/2)/\lfloor\log(n)\rfloor}} \right)  \\ 
& \le K_5 n^2 \exp\left(-(2q + 2)\log(n)\right)  \\
& = K_5 n^{-2q}, \quad n\ge n_3,
\end{align*} 
which implies $\mathds{P}(D^n)$ is an $o(n^{-q})$ function.  Thus, all four terms in the LHS of (\ref{eq:auxsum1}) are $o(n^{-q})$ functions, and so is $\mathds{P}(A^n)$.

Finally, let $n_4\ge n_3$ be such that $\delta_n<1-T$ for all $n\ge n_4$. Then, 
\begin{align*} 
	\mathds{P}\left(\chi^n_{n^2} < T\right)\le \mathds{P}\left(C^n_{\chi}\right) \mbox{ for all } n\ge n_4,
\end{align*} 
meaning that $(\ref{eq:rateaux1})$ is an $o(n^{-q})$ function. The proof of (\ref{eq:rate2}) is now complete.

\textbf{Proof of Part (ii).} Now, let $\{\rho_{\ell}\}_{\ell\ge 1}$ be a sequence with $\rho_{\ell}\downarrow 0$, and define 
\begin{align*} 
	E^\ell := \bigcap_{j=1}^\infty \bigcup_{n=j}^\infty \left\{\pi_2(J^n(s))\neq J(s)\mbox{ for some }s\in(T-\rho_\ell, T + \rho_\ell)\right\}.
\end{align*} 
Proving (\ref{eq:rateJ1}) is equivalent to showing that $\mathds{P}(\cap_{\ell=1}^\infty E^\ell) = 0$, which in turn is equivalent to proving that $\lim\limits_{\ell\rightarrow\infty}\mathds{P}(E^\ell)=0$. 

Define $\beta_0:=0$, $\beta^n_0:= 0$ for $n\ge 0$. For $k\ge 0$, let
\begin{align*}
\beta_{k+1}&:=\inf\left\{s>\beta_k: J(s^-)\neq J(s)\right\},\\
\beta_{k+1}^n&:= \inf \left\{s>\beta_k^n: \pi_2(J(s^-))\neq \pi_2(J(s))\right\},\quad n\ge 0.
\end{align*}
For any $a,b\in\mathds{R}$, define $M^0[a,b]: = M(b_+)-M(a_+)$; recall that $\mathcal{M}^0$ is a Poisson process of rate $\lambda_0/2$ defined in Section \ref{sec:ProofStrongConvMMBM}. Then,
\begin{align}
E^\ell& \subseteq \left\{M^0[T-2\rho_\ell, T+2\rho_\ell]>0\right\} \cup \left(\left\{M^0[T-2\rho_\ell, T+2\rho_\ell]=0\right\} \cap E^{\ell}\right).  
\end{align}
Note that $\mathds{P}(M^0[T-2\rho_\ell, T+2\rho_\ell]>0) \le 1-e^{-(\lambda_0/2)4\rho_\ell} \rightarrow 0$ as $\ell\rightarrow\infty$. Thus, in order to prove that $\lim\limits_{\ell\rightarrow\infty}\mathds{P}(E^\ell)= 0$, it is sufficient to show that 
	\begin{align} 
		\lim_{\ell\rightarrow\infty}\mathds{P}\left(\{M^0[T-2\rho_\ell, T+2\rho_\ell]=0\}\cap E^\ell\right)=0,
	\end{align} 
which we do next. A path inspection reveals that
\begin{align}
& \{M^0 [T-2\rho_\ell, T+2\rho_\ell]=0\}\cap E^\ell\nonumber\\
&\subseteq\left\{\bigcup_{k\ge 0}\left\{ \beta_k<T-2\rho_\ell <  T+2\rho_\ell<\beta_{k+1}\right\}\right\}\cap E^\ell \nonumber\\
&=\bigcup_{k\ge 0}\bigcap_{j=1}^\infty\bigcup_{n=j}^\infty \left\{\begin{array}{cc}\{\beta_k<T-2\rho_\ell<  T+2\rho_\ell<\beta_{k+1}, T-\rho_\ell < \beta^n_k \} \; \cup\\\{\beta_k<T-2\rho_\ell < T+2\rho_\ell<\beta_{k+1},  \beta^n_{k+1}<T+\rho_\ell \}\end{array}\right\}\nonumber\\
&\subseteq\bigcup_{k\ge 0}\bigcap_{j=1}^\infty\bigcup_{n=j}^\infty \left(\{|\beta_k - \beta_k^n|>\rho_\ell, \beta_k< T-2\rho_\ell \}\cup\{|\beta_{k+1} - \beta_{k+1}^n|>\rho_\ell, \beta^n_{k+1}< T+\rho_\ell  \}\right)\nonumber\\
\displaybreak
&\subseteq \left(\bigcap_{j=1}^\infty\bigcup_{n=j}^\infty \left\{\max_{k:\theta^n_k< T-2\rho_\ell}|\chi^n_k - \theta^n_k|>\rho_\ell\right\}\right) \; \cup \; \left(\bigcap_{j=1}^\infty \bigcup_{n=j}^\infty \left\{\max_{k:\chi^n_k< T+\rho_\ell}|\chi^n_k - \theta^n_k|>\rho_\ell\right\}\right)\nonumber\\
&\subseteq \left(\bigcap_{j=1}^\infty\bigcup_{n=j}^\infty \left\{\max_{1\le k\le n^2}|\chi^n_k - \theta^n_k|>\rho_\ell\right\}\right) \; \cup \; \left(\bigcap_{j=1}^\infty \bigcup_{n=j}^\infty\{ \theta^n_{n^2}<T-2\rho_\ell\}\right) \cup \left(\bigcap_{j=1}^\infty \bigcup_{n=j}^\infty\{\chi^n_{n^2}<T+\rho_\ell\}\right).
\label{eq:last2sum}
\end{align}

Since $\{\delta_n\}_{n\ge 1}$ is a sequence such that $\delta_n\downarrow 0$, then for each $\ell\ge 1$,
\begin{align*}
	%
\mathds{P}\left(\bigcap_{j=1}^\infty\bigcup_{n=j}^\infty\{\max_{1\le k\le n^2}|\chi^n_k - \theta^n_k| > \rho_\ell\}\right)\le \mathds{P}\left(\bigcap_{j=1}^\infty\bigcup_{n=j}^\infty\{\max_{1\le k\le n^2}|\chi^n_k - \theta^n_k| > 2\delta_n\}\right) = 0, 
\end{align*}
%
where the last equality follows from the fact that   
\begin{align*} 
	\mathds{P}\left(\max_{1\le k\le n^2}|\chi^n_k - \theta^n_k| > 2\delta_n\right)\le \mathds{P}(C^n_\chi\cup C^n_\theta)=o(n^{-q}),
\end{align*} 
and applying Borel-Cantelli (choosing, say, $q = 2$).  Similar arguments follow for the two other events in (\ref{eq:last2sum}). Thus, $\lim\limits_{\ell\rightarrow\infty}\mathds{P}(E^\ell)=0$ and so (\ref{eq:rateJ1}) follows.

\section{An application: First passage probabilities.} 
\label{sec:applications}
Theorem \ref{th:strongConvMMBM} implies that some first passage properties of $(\mathcal{R},\mathcal{J})$ can be analysed as the limiting first passage properties of $(\mathcal{R}^n,\mathcal{J}^n)$ as $n\rightarrow\infty$. In particular, for any Borel set $A\subset\mathds{R}$ define
\begin{align*}
\tau_A &:=\inf\left\{s\ge 0: R(s)\in A\right\}, \\
\tau^n_A &:=\inf\left\{s\ge 0: R^n(s)\in A\right\},\quad n\ge 0.
\end{align*}
Then, Theorem \ref{th:strongConvMMBM} implies that for any open set $A$ and $j\in\mathcal{S}$,
\begin{align*}
\tau_A= \lim_{n\rightarrow\infty}\tau_A^n\quad\mbox{a.s.},
\end{align*}
and on the event $\{\tau_A<\infty\}$,
\begin{align*}
\left\{J(\tau_A)=j\right\}=\bigcup_{i=0}^\infty\bigcap_{n=i}^\infty\{\pi_2(J^n(\tau_A))=j\}\quad\mbox{a.s.}.
\end{align*}

%
In the case $A$ takes the form $(-\infty,-x)$, for $x\ge 0$, we have the following.

\begin{Proposition} \label{prop:tauxtaunx1}For $x\ge 0$ and $n\ge 0$ define 
\[\tau_x:=\tau_{(-\infty,-x)}\quad\mbox{and}\quad\tau_x^n:=\tau_{(-\infty,-x)}^n.\] 
Then, for all $j\in\mathcal{S}$,
\begin{align}
	\{\tau_x<\infty, J(\tau_x)=j\} = \{\tau^n_x<\infty, \pi_2(J^n(\tau_x^n))=j\}.\label{eq:eventeqaux1}
\end{align}
\end{Proposition} 
\begin{proof}
Fix $x\ge 0$ and $n\ge 0$. Let 
\[N:=\sup\{k\ge 0: \tau_x\ge \theta^n_k\}.\]
This implies that $\tau_x\in[\theta^n_N,\theta^0_{N+1})$ on $\{N<\infty\}$, and since $\mathcal{J}$ is constant between the epochs $\{\theta^n_k\}_{k\ge 0}$, then $J(\theta^n_N)=J(\tau_x)$. Similarly, if we define 
\[N^n:=\sup\{k\ge 0: \tau_x^n\ge \chi^n_k\},\]
then  $\pi_2(J^n(\chi^n_{N^n}))=\pi_2(J^n(\tau_x^n))$ on $\{N^n<\infty\}$. Equations (\ref{eq:RnS1}) and (\ref{eq:RnS2}) imply that $N_n=N$, and since $J(\theta^n_k)=\pi_2(J^n(\chi^n_k))$ for all $k\ge 0$ (see (\ref{eq:sameJandJn})), then 
\[J(\tau_x)=J(\theta^n_N)=\pi_2(J^n(\chi^n_{N^n}))=\pi_2(J^n(\tau_x^n))\quad\mbox{on}\quad\{N<\infty\},\]
 and (\ref{eq:eventeqaux1}) follows.
\end{proof}
The following result describes one central first passage distributional property of our construction.
\begin{theorem}
	\label{th:quadratic1}
For $n\ge 0$, let $U_n$ denote the infinitesimal {generator} associated to {the process} $\{\pi_2(J^n(\tau^n_x))\}_{x\ge 0}$. Then, $U_n$ is a solution to the quadratic matrix equation 
\begin{align}
	\label{eq:quadratic2}
	X^2 + 2\Delta_{\mu}\Delta_{\sigma}^{-2}X + 2\Delta_{\sigma}^{-2}Q=0
\end{align}
where $\Delta_{\mu} = \mbox{diag}\{\mu_i: i\in\mathcal{S}\}$,  $\Delta_{\sigma}=\mbox{diag}\{\sigma_i: i\in\mathcal{S}\}$. Furthermore $U_n$ corresponds to the infinitesimal {generator} associated to $\{J(\tau_x)\}_{x\ge 0}$.
\end{theorem}
\begin{proof} 
Let $\Psi_n$ be the $p\times p$-dimensional matrix defined by
	\begin{align*} 
	(\Psi_n)_{ij}=\mathds{P}\left(\tau^n_0 <\infty, J^n(\tau_0)=(-,j) \mid R^n(0)=0, J^n(0)=(+,i)\right), \quad i,j\in\mathcal{S}.
	\end{align*} 
Define  $ 
	  \Delta_{r^n_+}  =\mbox{diag}\{r^n(+, i): i\in\mathcal{S}\}$,  $\Delta_{r^n_-}  =\mbox{diag}\{|r^n(-, i)|: i\in\mathcal{S}\}$,
and 
\begin{align*} 
	\left[\begin{array}{cc} 
	T_{++} & T_{+-} \\ 
	T_{-+} & T_{--}\end{array}\right] := 	\left[\begin{array}{rc} 
		-\lambda_n I & 2 Q + \lambda_n I \\
\lambda_n I & -\lambda_n I					\end{array}\right].
\end{align*} 
It is known \cite{bean2005algorithms} that $\Psi_n$is the minimal nonnegative solution to the Riccati matrix equation
\begin{align}
	\label{eq:Riccati1}
	\Delta_{r^n_+}^{-1}T_{++}\Psi_n + \Psi_n\Delta_{r^n_-}^{-1}T_{--} + \Psi_n\Delta_{r^n_-}^{-1}T_{-+}\Psi_n + \Delta_{r^n_+}^{-1}T_{+-} = 0, 	%
	\end{align}
and that
\begin{align*} 
	\mathds{P}(\tau^n_0 <\infty, J^n(\tau_x)=(-,j) \mid R^n(0)=0, J^n(0)=(-,i))=\bm{e}_i^{\top}e^{U_n x}\bm{e}_j,
\end{align*} 
where
\begin{align}
	U_n & = \Delta_{r^n_-}^{-1} (T_{--} + T_{-+}\Psi_n)
	=\lambda_n \Delta_{r^n_-}^{-1} (\Psi_n - I),\label{eq:rhsUn1}
\end{align}
with $\bm{e}_i$ being the $i$th unit column vector.

Premultiplying (\ref{eq:Riccati1}) by $\Delta_{r^n_-}^{-1}$ and commuting $\Delta_{r^n_-}^{-1}$ with $\Delta_{r^n_+}^{-1}$ give
\begin{align*}
-\lambda_n \Delta_{r^n_+}^{-1} \Delta_{r^n_-}^{-1}\Psi_n -\lambda_n \Delta_{r^n_-}^{-1}\Psi_n\Delta_{r^n_-}^{-1} + \lambda_n\Delta_{r^n_-}^{-1}\Psi_n\Delta_{r^n_-}^{-1}\Psi_n + \Delta_{r^n_+}^{-1}\Delta_{r^n_-}^{-1}(2Q + \lambda_nI) = 0,
\end{align*}
which leads to 
%
%
%
%
$
(\Delta_{r^n_-}^{-1}-\Delta_{r^n_+}^{-1})U_n+ \lambda_n^{-1}U_n^2 + 2\Delta_{r^n_+}^{-1}\Delta_{r^n_-}^{-1}Q = 0.
$
As $
\Delta_{r^n_-}^{-1}-\Delta_{r^n_+}^{-1} = 2\lambda_n^{-1}\Delta_{\mu}\Delta_{\sigma}^{-2}$ and $ 
\Delta_{r^n_+}^{-1}\Delta_{r^n_-}^{-1}  = \lambda_n^{-1}\Delta_{\sigma}^{-2}$, we obtain 
\begin{align}
	\label{eq:Quadratic1}
U_n^2 + 2\Delta_{\mu}\Delta_{\sigma}^{-2}U_n + 2\Delta_{\sigma}^{-2}Q=0.
\end{align}
That $U_n$ is also the infinitesimal generator of $\{J(\tau_x)\}_{x\ge 0}$ follows from Proposition \ref{prop:tauxtaunx1}.
\end{proof}

%
\begin{Remark} Theorem \ref{th:quadratic1} provides a novel understanding of the classic quadratic matrix equation associated to the {down-crossing records of an MMBM} (see \cite{Asmussen:1995jm}). Indeed, to compute the infinitesimal generator solution of (\ref{eq:quadratic2}) (which is unique by \cite{latouche2015morphing}), we can instead compute the minimal nonnegative solution to the Riccati matrix equation (\ref{eq:Riccati1}), say $\Psi_n$. The solution of (\ref{eq:quadratic2}) is then given by $U_n$ as defined in (\ref{eq:rhsUn1}). A comparable result is that of \cite{latouche2015morphing}, where the authors construct a sequence of matrices $\{U^*_n\}_{n\ge 0}$ that is shown to converge to $U$. One advantage of our construction is that each element of the sequence $\{U_n\}$ obtained through Theorem \ref{th:quadratic1} is identical to $U$.
\end{Remark}
\section*{Acknowledgements.}
Both authors are affiliated with Australian Research Council (ARC) Centre of Excellence for Mathematical and Statistical Frontiers (ACEMS).
\section*{Appendix.}
The following are some standalone results used in Sections \ref{sec:ProofStrongConvMMBM} and \ref{sec:rate}.
\begin{theorem}\label{th:reverseUnif}
Let $\mathcal{A}=\{A(t)\}_{t\ge 0}$ be a Poisson process of parameter $\lambda_a>0$, and $\mathcal{X}=\{X(n)\}_{n\ge 0}$ an independent discrete-time Markov chain with state space $\mathcal{S}$ and transition probability matrix $P$. Define the Markov jump process $\mathcal{J}=\{J(t)\}_{t\ge 0}$ be
\[J(t)= X(A(t)),\quad t\ge 0.\]
Let $\mathcal{B}=\{B(t)\}_{t\ge 0}$ be an independent Poisson process of parameter $\lambda_b>0$. Define $\mathcal{C}$ to be the superposition of the Poisson processes $\mathcal{A}$ and $\mathcal{B}$, and denote by $\{\tau_k\}_{k\ge 0}$ the arrival times of $\mathcal{C}$. If we let
\[Y(n)=J(\tau_n), \quad n\ge 0,\]
then the process $\mathcal{Y}=\{Y(n)\}_{n\ge 0}$ is a Markov chain with transition probability matrix given by
\begin{equation}\label{eq:transitionmatrix1}\frac{\lambda_a}{\lambda_a+\lambda_b}P + \frac{\lambda_b}{\lambda_a+\lambda_b}I.
\end{equation}
\end{theorem}
\begin{proof}
First, we show that $\mathcal{Y}$ is a Markov process. Let $i\in\mathcal{S}$ and $k\ge 1$. Then,
\begin{align*}
\mathds{P}&\left(Y(k)=i \mid Y(0), Y(1), \dots, Y(k-1)\right)\\
& = \mathds{P}(J(\tau_{k})=i \mid J(\tau_0), J(\tau_{1}), \dots, J(\tau_{k-1}))\\
& = \mathds{P}(J(\tau_{k})=i \mid J(\tau_{k-1}))\quad\mbox{(Strong Markov property of $\mathcal{J}$)}\\
& = \mathds{P}(Y(k)=i \mid Y(k-1)),
\end{align*}
so that the Markov property holds.

Next, let $\mathcal{C}^*$ be the \emph{marked} Poisson process with arrivals corresponding to the superposition of $\mathcal{A}$, arrivals which we mark with an $a$, and $\mathcal{B}$, arrivals which we mark with a $b$. The $k$th arrival of $\mathcal{C}^*$ occurs at $\tau_k$ carrying a mark, say $m_k\in\{a,b\}$. Then,
\begin{align*}
& \mathds{P}(Y(k)=j \mid Y(k-1)=i) \\
& = \mathds{P}(Y(k)=j, m_k=a \mid Y(k-1)=i) + \mathds{P}(Y(k)=j, m_k=b \mid Y(k-1)=i)\\
& = \mathds{P}(Y(k)=j \mid Y(k-1)=i, m_k=a)\mathds{P}(m_k=a  \mid Y(k-1)=i)\\ 
&\quad + \mathds{P}(Y(k)=j \mid Y(k-1)=i, m_k=b)\mathds{P}(m_k=b \mid Y(k-1)=i).
\end{align*}
The event $\{m_k=a\}$ is clearly independent from $\{Y(k-1)=i\}$: the mark of a given Poisson arrival is independent of the history of the previous arrivals. Thus,
\[\mathds{P}(m_k=a \mid Y(k-1)=i) = \mathds{P}(m_k=a) = \frac{\lambda_a}{\lambda_a+\lambda_b}.\]
Similarly, 
\[\mathds{P}(m_k=b \mid Y(k-1)=i) = \mathds{P}(m_k=b) = \frac{\lambda_b}{\lambda_a+\lambda_b}.\]
Next, since $\mathcal{J}$ only (possibly) jumps at arrival times marked with $a$, then
\[\mathds{P}(Y(k)=j \mid Y(k-1)=i, m_k=b) = \delta_{ij},\]
where $\delta_{ij}$ denotes the Kronecker delta. Finally, since $\mathcal{J}$ is piecewise constant between the arrival times $\{\tau_k\}_k$, then 
\[\{Y(k-1)=i\} = \{J(\tau_{k-1})= i\}=\{J(\tau_{k}^-)= i\}.\]
This implies that
\begin{align*}
\mathds{P}(Y(k)=j \mid Y(k-1)=i, m_k=a) & = \mathds{P}(J(\tau_{k})= j \mid J(\tau_{k}^-)= i, m_k=a) = p_{ij}.
\end{align*}
Consequently,
\begin{align*}
\mathds{P}(Y(k)=j \mid Y(k-1)=i) = p_{ij}\frac{\lambda_a}{\lambda_a+\lambda_b} + \delta_{ij}\frac{\lambda_a}{\lambda_a+\lambda_b}
\end{align*}
and the proof is complete.
\end{proof}

\begin{Lemma}
	\label{lem:cmomentErl}
For $a\in \mathds{N}_{+}\backslash\{1\}$ and $b>0$, let $Y\sim\mbox{\mbox{Erlang}}(a,b)$.  
Then
\begin{align}
	\label{eq:cmomentErl}
\mathds{E}\left[(Y-\mathds{E}[Y])^k\right] \le \frac{k!\sqrt{a}}{b^j} \frac{\sqrt{a}^{k+1}-1}{\sqrt{a}-1} \quad \mbox{ for } k\in\mathds{N}_{+}. 
\end{align}
\end{Lemma}
\begin{proof}
W.l.o.g. suppose that $b=1$. Equation (\ref{eq:cmomentErl}) can be rewritten as
\begin{align}
	\label{eq:cmomentErl2}
\mathds{E}\left[(Y-\mathds{E}[Y])^k\right] \le {k!}\sum_{j=1}^k\sqrt{a}^j.
\end{align}
We use induction to prove that (\ref{eq:cmomentErl2}) holds. First, since $\mathds{E}[Y-\mathds{E}[Y]]=0< 1!\sqrt{a}$, the case $k=1$ holds trivially. Now, suppose (\ref{eq:cmomentErl}) holds for all $k\in\{1,2,\dots, k_0\}$ for some $k_0\ge 1$. By \cite[third formula on p.704]{willink2003relationships}, 
\begin{align}
\mathds{E}\left[(Y-\mathds{E}[Y])^{k_0+1}\right] & = k_0!a\sum_{i=0}^{k_0-1} \frac{\mathds{E}\left[(Y-\mathds{E}[Y])^{i}\right]}{i!} \nonumber \\
& =k_0!a\left[1 + \sum_{i=2}^{k_0-1}\frac{\mathds{E}[(Y-\mathds{E}[Y])^{i}]}{i!}\right]. 	\label{eq:willink}
\end{align}
Using the induction hypothesis on the RHS of (\ref{eq:willink}), we get 
\begin{align*}
\mathds{E}[(Y-\mathds{E}[Y])^{k_0+1}] & \le k_0!a\left[1 + \sum_{i=2}^{k_0-1}\sum_{j=1}^i\sqrt{a}^j\right] \le k_0!a\sum_{i=1}^{k_0-1}\sum_{j=1}^i\sqrt{a}^j\\
&= k_0!a\sum_{j=1}^{k_0-1}\sum_{i=j}^{k_0-1}\sqrt{a}^j= k_0!a\sum_{j=1}^{k_0-1}(k_0-j)
\sqrt{a}^j\\
& \le (k_0+1)!a\sum_{j=1}^{k_0-1}\sqrt{a}^j\le (k_0+1)!\sum_{j=1}^{k_0+1}\sqrt{a}^j,
\end{align*}
which proves (\ref{eq:cmomentErl2}) and thus (\ref{eq:cmomentErl}).
\end{proof}
\begin{Lemma}
	\label{lem:boundforMMBM}
Let $(\mathcal{R}, \mathcal{J}) = \{(R(t), J(t))\}_{t\ge 0}$ be a Markov-modulated Brownian motion defined as in (\ref{eq:Rt1}). Then, for any $i\in\mathcal{S}$, $t > 0$ and ${a > \mu_{\max} t}$,
\begin{align}
\mathds{P}\left(\left.\sup_{0 \le s\le t}|R(s)|> a \; \right| \; J(0) = i \right)\le \frac{2}{{\sqrt{2 \pi}}}{\frac{\sqrt{\sigma_{\max}t}}{a - \mu_{\max}t}}\exp\left(-\frac{(a-\mu_{\max}t)^2}{2\sigma_{\max}t}\right),
\end{align}
where $\mu_{\max}:=\max_{i\in\mathcal{S}}|\mu_i|$ and $\sigma_{\max}:=\max_{i\in\mathcal{S}}\sigma_i.$
\end{Lemma}
\begin{proof}
Let $\{W(t)\}_{t\ge 0}$ be a standard Brownian motion, independent from $(\mathcal{R},\mathcal{J})$. A standard bound for the Brownian motion 
gives us for $b > 0$
\begin{align*} 
	\mathds{P}\left(\sup_{0 \le s\le t}|W(s)|> b\right) 
	& =  2 \int_b^{\infty} \frac{1}{\sqrt{2\pi t}}e^{-x^2/ 2t} \dd x  \\
	& \leq 2 \int_{b}^{\infty} \frac{x / t}{\sqrt{2 \pi t}} e^{-x^2/ 2t} \dd x \\
	&   \le \frac{2}{\sqrt{2 \pi}} \left(\frac{\sqrt{t}}{b}\right)e^{-b^2/2t}.
\end{align*} 
Note that $\mathcal{R}$ is identically distributed to $\{W(I^\sigma_t) + I^\mu_t\}_{t\ge 0}$, where $	I^\sigma_t := \int_0^t\sigma_{J(s)}\dd s$ and $I^{\mu}_t :=\int_0^t\mu_{J(s)}\dd s.$ 
This implies that
\begin{align*}
\mathds{P}&\left(\left.\sup_{0 \le s\le t}|R(s)|> a\;\right| \; J(0)=i\right) = \mathds{P}\left(\left.\sup_{0 \le s\le t}|W\left(I^\sigma_s\right) + I^\mu_s|> a\;\right| \; J(0)=i\right)\\
& \le \mathds{P}\left(\left.\sup_{0 \le s\le t}|W(I^\sigma_s)| > a - \sup_{0 \le s\le t}|{I^\mu_s}|\;\right| \; J(0)=i\right)\\
& =  \mathds{E}\left(\left.\mathds{P}\left(\left.\sup_{0 \le s\le t}\left|W(I^\sigma_s)\right| > a - \sup_{0 \le s\le t}|{I^\mu_s}| \; \right| \; J(0)=i, \{I^\sigma_s\}_{0\le s\le t}, \{I^\mu_s\}_{0\le s\le t} \right)\;\right|\;J(0)=i\right) \\
& \le \mathds{E}\left(\left.\frac{2}{{\sqrt{2\pi}}}{\left(\frac{\sqrt{\sup_{0 \le s\le t}|{I^\sigma_s}|}}{a-\sup_{0 \le s\le t}|{I^\mu_s}|}
\right)} \exp\left(-\frac{(a-\sup_{0 \le s\le t}|{I^\mu_s}|)^2}{2\sup_{0 \le s\le t}|{I^\sigma_s}|}\right)\;\right| \; J(0)=i\right)\\
& \le \mathds{E}\left(\left.\frac{2}{{\sqrt{2 \pi}}}{\frac{\sqrt{\sigma_{\max}t}}{a - \mu_{\max}t}}\exp\left(-\frac{(a-\mu_{\max}t)^2}{2\sigma_{\max}t}\right)\;\right| \; J(0)=i\right) \\
& = \frac{2}{{\sqrt{2 \pi}}}{\frac{\sqrt{\sigma_{\max}t}}{a - \mu_{\max}t}}\exp\left(-\frac{(a-\mu_{\max}t)^2}{2\sigma_{\max}t}\right),
\end{align*}
which completes the proof.
\end{proof}

\bibliographystyle{abbrv}
\end{document}